\documentclass[letterpaper]{amsart}

\usepackage[margin=1in]{geometry}
\usepackage[colorlinks=true, citecolor=blue]{hyperref}
\usepackage{graphicx,amsmath,amssymb,amsthm,paralist,color,tikz-cd}
\usepackage{titlesec}

\titleformat{\section}
  {\centering\normalfont\scshape}{\thesection}{1em}{}
\titleformat{\subsection}
  {\normalfont\scshape}{\thesubsection}{.8em}{}
\titleformat{\subsubsection}
  {\normalfont}{\thesubsubsection}{.7em}{}

\theoremstyle{plain} % all numbering is sequential based only on the section
\newtheorem{thm}{Theorem}[section]

\newtheorem{lemma}[thm]{Lemma}
\newtheorem{cor}[thm]{Corollary}

\theoremstyle{definition} % same numbering as above

\newcommand{\C}{\mathbb{C}}

\newcommand{\mc}{\mathcal}

\renewcommand{\d}{\delta}
\newcommand{\e}{\varepsilon}
\renewcommand{\l}{\lambda}

\newcommand{\vp}{\varphi}
\renewcommand{\t}{\tau}

\newcommand{\set}[1]{\left\{#1\right\}}

\renewcommand{\o}{\overline}

\renewcommand{\r}{\rightarrow}

\newcommand{\pa}[2]{\frac{\partial #1}{\partial #2}}

\renewcommand{\Re}{\operatorname{Re}}

\newcommand{\pair}[1]{\left(#1\right)}

\renewcommand{\mod}[1]{\,(\operatorname{mod}{#1})}

\def\XXint#1#2#3{{\setbox0=\hbox{$#1{#2#3}{\int}$ }
\vcenter{\hbox{$#2#3$ }}\kern-.5\wd0}}

\begin{document}
\title{Hybrid Level Aspect Subconvexity for $GL(2)\times GL(1)$ Rankin-Selberg
  $L$-Functions}
\author{Keshav Aggarwal, Yeongseong Jo, and Kevin Nowland}
\address{Department of Mathematics, The Ohio State University, 100 Math Tower,
231 West 18th Avenue, Columbus, OH 43210-1174}
\email{aggarwal.78@osu.edu}
\address{Department of Mathematics, The Ohio State University, 100 Math Tower,
231 West 18th Avenue, Columbus, OH 43210-1174}
\email{jo.59@osu.edu}
\address{Department of Mathematics, The Ohio State University, 100 Math Tower,
231 West 18th Avenue, Columbus, OH 43210-1174}
\email{nowland.159@osu.edu}


\begin{abstract}
 Let $M$ be a squarefree positive integer and $P$ a prime number coprime to
  $M$ such that $P\sim M^\eta$ with  $0 < \eta < 2/5$. We simplify the proof of subconvexity bounds for
  $L(\frac{1}{2},f\otimes\chi)$ when $f$ is a primitive holomorphic cusp form of level 
  $P$ and $\chi$ is a primitive Dirichlet character modulo $M$. These bounds are 
  attained through an unamplified second moment method using a modified version 
  of the delta method due to R. Munshi. The technique is similar to that used by
  Duke-Friedlander-Iwaniec save for the modification of the delta method.
\end{abstract}

\maketitle

\section{Introduction}
In studying the subconvexity problem for character twists of holomorphic modular 
forms of full level, Duke-Friedlander-Iwaniec \cite{dfi} introduced a simple yet 
powerful decomposition of the delta symbol which detects when an integer is zero. 
The starting point for their method was an amplified second moment 
average over primitive Dirichlet characters of a given level. The subconvexity 
problem for twisted $L$-functions $L(s,f\otimes \chi)$ in the conductor-aspect 
was also solved in their paper for the first time for $f$ a holomorphic cusp form
of level one with subconvexity exponent $1/2-1/22$ (following the computation of 
section 4.3 in \cite{mi07} for example).

Recent works achieving hybrid subconvexity bounds for Rankin-Selberg convolution $L$-functions of
large level include Kowalski-Michel-Vanderkam \cite{kmvdk}, Michel \cite{mi04} and Harcos-Michel
\cite{hami06}. Michel-Venkatesh \cite{mive10} solve the subconvexity problem for the $L$-functions
of $GL(1)$ and $GL(2)$ automorphic representations over a fixed number field, uniformly in all
aspects. Holowinsky-Munshi \cite{homu13} prove a hybrid level aspect bound for the $L$-function
coming from the convolution of two holomorphic modular forms of nontrivial levels, one being 
squarefree and the other being prime. 
Z. Ye \cite{zh14} relaxed the level conditions in \cite{homu13} to both levels square-free. Moreover, 
he used a Large Sieve inequality to establish a subconvexity bound for the full range of levels when both forms are holomorphic.
The works of Holowinsky-Munshi and Z. Ye relied on an application of Heath-Brown's refinement 
\cite{heathbrown} of the classical delta method due to Duke-Friedlander-Iwaniec \cite{dfi}. Browning-Munshi in \cite{brmu13}
introduce a modification of the delta method with factorization moduli to
obtain a structural advantage.

In this paper we follow the work of Duke-Friedlander-Iwaniec \cite{dfi} who studied the
Rankin-Selberg  convolution of a primitive Dirichlet 
character with a holomorphic modular form for the full modular group.
Their paper, as ours, uses a second moment average over primitive Dirichlet
characters of a given level. In this paper, we allow for holomorphic forms
with a range of permissible but nontrivial levels relative to the conductor of 
the Dirichlet character.  It is here that we make use of the modified delta method with a conductor lowering trick, 
based on the work of Munshi 
in \cite{munshi}.

Let $f$ be a primitive holomorphic cusp form of level $P$ and $\chi$ 
a primitive Dirichlet character modulo $M$. The Rankin-Selberg convolution 
$L(s,f\otimes\chi)$ is given by
\[ L(s,f\otimes\chi) = \sum_{n\geq1} \frac{\lambda_f(n)\chi(n)}{n^s}, \]
at least for $\Re(s)$ sufficiently large.
Subconvexity for these Rankin-Selberg $L$-functions
has already been established in \cite{co03,dfi,ha03,mi07}. The main point of this paper, 
however, is to demonstrate how using a modified application of the delta method simplifies arithmetic 
structure and lengthens the admissible hybrid range of the level parameters $M$ and $P$.
Specifically, we present a simple method which ends with a trivial application of the Weil bound 
for Kloosterman sums and establishes subconvexity for the hybrid range $P=M^{\eta}$
for $0<\eta<2/5$. Using the classical delta method without the conductor lowering trick 
and following the same process, one would obtain a hybrid range of $0<\eta<2/7$. 
In the meanwhile, Blomer and Harcos in theorem 2 of \cite{blha14} establish the following hybrid estimate
\[
  L\left(\frac{1}{2}, f \otimes \chi \right) \ll_{{k,\e}} P^{\frac{1}{4}+\e}M^{\frac{3}{8}+\e}
  +P^{\frac{1}{2}+\e}M^{\frac{1}{4}+\e}.
\]
For $P=M^{\eta}$, we obtain that
\[
  L\left(\frac{1}{2}, f \otimes \chi \right) \ll_{{k,\e}} \mc Q^{\frac{1}{4}+\e}
  \left( \mc Q^{-\frac{1}{8(2+\eta)}-\frac{1-\eta}{4(2+\eta)}} \right),
\]
where $\mc Q=\mc Q(f \otimes \chi)=PM^2$ is the size of the conductor of $L$-function $L(s,f\otimes \chi)$. 
The hybrid range $0 <  \eta < 1$ of Blomer and Harcos is stronger than
our range $0 < \eta <2/5$. However we emphasize that our technique does not require amplification or Large Sieve inequality
and recently this method is adopted in \cite{HoZh} to extend a hybrid subconvexity range bound for $L(1/2,g \otimes h)$ 
where $g$ is a primitive holomorphic cusp form of level $M$
 and $h$ is a primitive either holomorphic or Mass cusp form of level $P$ with $(M,P)=1$, $M$
 a squarefree integer, and $P$ a prime. 
 
Of course, one has the ability to push the analysis further, in either method, by analyzing the resulting 
sum of Kloosterman sums through Large Sieve inequality similar to the work of Z. Ye. Again,
however, there is an advantage in the modified delta method in that we obtain sums of 
standard Kloosterman sums for the group $\Gamma_0(P)$ associated to the cusp at $\infty$.
Without the modified delta method, one would instead get Kloosterman sums associated to 
the cusps $0$ and $\infty$ and then more work is required 
(using the work of Deshoullier-Iwaniec \cite{desiwa} for example). 

We provide a sketch of
these arguments below and note that our methods may also be applied to analogous Rankin-Selberg convolutions.

\subsection{Holomorphic cusp forms}
Let $P>0$ be a prime number and $k>0$ an even integer. Let $S_k(P)$ be the
linear space of holomorphic cusp forms of weight $k$, level $P$, and trivial
nebentypus. We let $\Gamma_0(P)$ be the Hecke congruence subgroups defined by
\[
  \Gamma_0(P)=\left\{ \begin{pmatrix} a& b\\c&d\end{pmatrix}
  \in SL(2,\mathbb{Z})\; \middle| \; c \equiv 0 \mod 
  P \right\}
\]
If $f\in S_k(P)$, then $f:\mathbb{H}\r\C$  is holomorphic and 
satisfies
\[
  f(g z) = (cz+d)^kf(z)
\]
for every $g=\pair{\begin{matrix}a&b\\c&d\end{matrix}}\in\Gamma_0(P)$ acting
as a linear fractional transformation on $z$ in the upper half-plane 
$\mathbb{H}$. Additionally, $f$ vanishes at every cusp. Such an $f$ has Fourier 
expansion
\[
  f(z) = \sum_{n=1}^\infty \psi_f(n)n^{\frac{k-1}{2}}e(nz),
\]
where $e(x):=\exp(2\pi ix)$, and the Fourier coefficients $\psi_f(n)$ satisfy
\[
  \psi_f(n)\ll_f \t(n) n^\e,
\]
with $\t(n)$ the divisor function for $\e>0$ arbitrary. This was proved by 
Deligne \cite{del74}.

$S_k(P)$ is a finite dimensional Hilbert space with respect to the Petersson
inner product. $S_k(P)$ has an orthogonal basis $B_k(P)$ which consists of 
eigenfunctions of all Hecke operators $T_n$ such that $(n,P)=1$, where $T_n$
acts on $f$ by the formula
\[
  T_nf(z) = \frac{1}{\sqrt{n}}\sum_{\substack{ad=n\\(a,P)=1}}
  \pair{\frac{a}{d}}^{k/2}\sum_{b\mod{d}}f\pair{\frac{az+b}{d}}
  =: \l_f(n)f.
\]
Such an $f$ is called a \emph{Hecke eigen cusp form}. The Hecke operators are 
multiplicative and satisfy
\[
  \psi_f(m)\l(n) = \sum_{d\mid(m,n)}\psi_f\pair{\frac{mn}{d^2}}
\]
for $n$ coprime with $P$. In particular, $\psi_f(1)\l_f(n)=\psi_f(n)$
for $(n,P)=1$. We normalize such that $\psi_f(1)=1$ and we have
$\l_f(n)=\psi_f(n)$ for $(n,P)=1$. There exists a subset $B^*_k(P)$ of 
$B_k(P)$ of \emph{newforms} or  \emph{primitive holomorphic cusp forms} which are eigenfunctions of all Hecke
operators $T_n$ for $n\geq1$ with $\l_f(n)=\psi_f(n)$. Let $f \in B^*_k(P)$
be a newform and $\chi$ a primitive Dirichlet character modulo $M$ with $(P,M)=1$.
Let $f \otimes \chi$ be a twisted modular form on $\mathbb{H}$ given by the Fourier 
expansion
\[
   (f \otimes \chi)(z)=\sum_{n=1}^{\infty} \chi(n) \lambda_f(n)n^{\frac{k-1}{2}} e(nz).
\]
Then $f \otimes \chi$ is a newform of level $PM^2$.

\subsection{Rankin-Selberg \texorpdfstring{$L-$}-functions}
Let $f\in B_k^*(P)$ be a newform and $\chi$ a primitive Dirichlet character
modulo $M$ where $M$ and $P$ are coprime, $M$ is squarefree, and $P$ is a
prime. Then the Rankin-Selberg convolution $L$-function associated to
$f\otimes\chi$ is 
\[
  L(s,f\otimes\chi) = \sum_{n=1}^\infty \frac{\l_f(n)\chi(n)}{n^s}.
\]
The associated completed $L$-function is
\[
  \Lambda(s,f\otimes\chi) = \mc Q^sL_\infty(s,f\otimes\chi)L(s,f\otimes \chi),
\]
where $\mc Q=\mc Q(f\otimes\chi)= PM^2$ and the local factor at infinity 
$L_\infty$ is a product of gamma functions. The approximate functional equation 
shows that the special value 
$L(1/2, f\otimes \chi)$ is given by
\[
  L(1/2, f\otimes \chi)=\sum_{n=1}^{\infty}
    \frac{\lambda_{f}(n)\chi(n)}{\sqrt{n}}
    V\pair{\frac{n}{\sqrt{\mathcal{Q}}}}
    +\epsilon(f \otimes \chi)\sum_{m=1}^{\infty}
    \frac{\lambda_f(m)\overline{\chi(m)}}{\sqrt{m}}
    V\pair{\frac{m}{\sqrt{\mathcal{Q}}}}
\]
(see chapter 5 of \cite{iwko04}) where $V$ is a smooth function with rapid decay 
at infinity, and for any positive integer $A$, 
the derivatives of $V(y)$ satisfy
\[
  y^jV^{(j)}(y) \ll_{k} \mathcal{Q}^{\e}(1+y)^{-A}\log(2+y^{-1})
\]
for any $\e>0$. We also have the asymptotic
\[
  V(y)=1 + \mathnormal{O}
  \left(\left(\frac{y}{\sqrt{\mathcal{Q}}}\right)^{\alpha}\right)
\]
for $\alpha > 0$.
Let $h$ be a smooth function which is compactly supported on $[1/2,5/2]$ 
with bounded derivatives and suppose that $X$ runs over $2^{\nu}$ with 
$\nu=-1,0,1,2,3,\ldots$. Applying a smooth partition of unity and 
the asymptotic for $V$, we are left with
\[
  \left| L(1/2, f \otimes \chi) \right| \ll_{k,\e} \mathcal{Q}^\e 
  \left\{ \underset{\mc{Q}^{\frac{1}{2}-\d}
  \leq X \leq \mc{Q}^{\frac{1}{2}+\e}}{\max} |L_{f \otimes \chi}(X)| +
  \mc{Q}^{\frac{1}{4}-\frac{\d}{2}}\right\},
\]
where
\begin{equation} \label{twisted-Lfunction}
  L_{f \otimes \chi}(X) := 
  \sum_{n} \frac{\l_{f}(n)\chi(n)}{\sqrt{n}}h\left(\frac{n}{X}\right).
\end{equation}

\section{Statement of main results}
 We prove the hybrid range and subconvexity bounds for $L(1/2,f \otimes \chi)$. 
 We average over primitive 
 characters modulo $M$ through a second moment method to achieve subconvexity 
 bounds. 

\begin{thm} [Second Moment] 
  Let $P$ and $M$ be coprime positive integers with $P$ prime and $M$ 
  squarefree. Let $k$ be a fixed positive even integer. Set $\mathcal{Q}=PM^2$.
  Let $h$ be a smooth function with support in $[1/2,5/2]$ and 
  bounded derivatives. Let $\e, \d >0$ and choose any $X$ such that
  $\mathcal{Q}^{\frac{1}{2}-\d}\leq X \leq \mathcal{Q}^{\frac{1}{2}+\e}$. 
  If $f \in B^*_k(P)$ and $\chi$ is a primitive Dirichlet character 
  modulo $M$, then
  \begin{equation}\label{sec-moment}
    \frac{1}{\varphi^\star(M)}\underset{\chi\mod{M}}{{\sum}^{\star}}
      |L_{f \otimes \chi}(X)|^2 
    \ll_{k,\e} \mathcal{Q}^{\e}
       \left(
       1+\frac{\mathcal{Q}^{\frac{1}{2}}}{M}\cdot
       \frac{P^{\frac{5}{8}+\frac{\d}{4}}}{M^{\frac{1}{4}-\frac{\d}{2}}}
       \right).
  %    \left\{1+\frac{\mathcal{Q}^{\frac{1}{2}}}{M}\cdot\pair{
  %    \frac{P^{\frac{5}{8}+\frac{\d}{4}}}{M^{\frac{1}{4}-\frac{\d}{2}}}
  %    +\frac{1}{P^2}}\right\}.
  \end{equation}
\end{thm}
$\sum^{\star}$ means summation over primitive characters 
or over integers coprime with the specified modulus, and $\vp^{\star}$ is 
the number of primitive multiplicative characters modulo $M$.  
We apply the theorem below to $F(x,y)=h(x)h(y)$, $f_1=f_2$, and $X=Y$ for any 
$\mathcal{Q}^{\frac{1}{2}-\d}\leq X \leq \mathcal{Q}^{\frac{1}{2}+\e}$ to obtain
a second moment bound.
The first and second terms on the right hand side in \eqref{sec-moment} come from 
a zero shift and theorem \ref{shifted-sums-thm} on shifted convolution sums, 
respectively. Notice that if $P=1$, then this matches the bound in \cite{dfi} 
without amplification.

\begin{thm} [Shifted Convolution Sums]\label{shifted-sums-thm}
  Let $f_1,f_2$ be newforms in $B_k^*(P)$. Let $M$ be a positive squarefree
  integer coprime with $P$. Let $r\neq0$ be an integer coprime with $P$.
  Let $X,Y\geq1$ and $F$ a smooth function supported on $[1/2,5/2]\times
  [1/2,5/2]$ satisfying
  \[
    x^iy^j\pa{^i}{x^i}\pa{^j}{y^j}F\pair{\frac{x}{X},\frac{y}{Y}}
    \ll ZZ_x^iZ_y^j,
  \]
  for $Z>0$ and $Z_x,Z_y\geq1$. Then
  \begin{align}
    \sum_n
    \sum_{m=n+rM} \frac{\l_{f_1}(n)\l_{f_2}(m)}{\sqrt{nm}}
      F\pair{\frac{n}{X},\frac{m}{Y}} \ll_{k,\e}
      (PMr)^\e Z(Z_xZ_y)^{\frac{1}{2}}\max\set{Z_x,Z_y}^2 P^{\frac{3}{4}}    \frac{\max\set{X,Y}^{\frac{3}{4}}}{(XY)^{\frac{1}{2}}}.
      \label{shifted}
  \end{align}
\end{thm}

The second moment bound which is better by a power of 
$\mc Q$ than the convexity bound means that we must have a bound which is better than
\[
  \frac{1}{\varphi^\star(M)}\underset{\chi\mod{M}}{{\sum}^{\star}}
  |L_{f \otimes \chi}(X)|^2 
  \ll \frac{\mathcal{Q}^{\frac{1}{2}+\e}}{M}.
\]
Therefore the bounds in \eqref{sec-moment} induce 
the hybrid subconvexity range $P\sim M^\eta$ for $0<\eta<2/5$.

\begin{cor}[Subconvexity]
	Let $f$ and $\chi$ be as above with $\eta=\frac{\log P}{\log M}$. Then
    \[
    	L(1/2,f\otimes\chi) \ll_{k,\e}
        \frac{\mc Q^{\frac{1}{4}+\e}}{\mc Q^{\frac{2-5\eta}{20(2+\eta)}}}.
 %       \mc Q^{\frac{1}{4}+\e}\pair{\frac{1}{\mc Q^{\frac{2-5\eta}{20(1+2\eta)}}}
%        +\frac{1}{\mc Q^{\frac{\eta}{\eta+2}}}}.
    \]
    This produces a subconvex bound for $0<\eta<2/5$.
\end{cor}

\begin{proof}
From the reduction in section 1, 
\begin{equation}\label{aprox-fcn-equ}
  |L(1/2,f\otimes\chi)| \ll_{k,\e} \mc Q^\e\left\{\mc Q^{\frac{1}{4}-\frac{\d}{2}}
  + \max_{\mc Q^{\frac{1}{2}-\d}\leq X \leq \mc Q^{\frac{1}{2}+\e}}
  |L_{f\otimes\chi}(X)|\right\}.
\end{equation}
To bound this, weaken 
the bound in the above theorem by just taking the second term in 
\eqref{sec-moment} as the first provides saving for any $P$ and $M$ having size. 
This gives 
\[
%\begin{equation} \label{new-Lbound}
  |L_{f\otimes\chi}(X)| \ll \mc Q^{\frac{1}{4}+\e}
  \frac{P^{\frac{5+2\d}{16}}}{M^{\frac{2-4\d}{16}}}.
  %\pair{\frac{P^{\frac{5+2\d}{16}}}{M^{\frac{2-4\d}{16}}}+\frac{1}{P}}.
%\end{equation}
\]
Equate the first term and second terms on the right hand side of \eqref{aprox-fcn-equ}
while replacing $P$ and $M$ with powers of $\mc Q$ by using $P=M^\eta$. 
We have $P=\mathcal{Q}^{\frac{\eta}{2+\eta}}$ and $M=\mathcal{Q}^{\frac{1}{2+\eta}}$.
Therefore the optimal choice of $\delta$ satisfies
\[
  -\frac{\delta}{2}=\frac{5+2\delta}{16}\cdot\frac{\eta}{2+\eta}
  -\frac{2-4\d}{16}\cdot\frac{1}{2+\eta}.
\]
Then the saving $\d$ is then explicitly calculated to be $\d=\frac{2-5\eta}{10(2+\eta)}$. 
\end{proof}

\section{Proof Sketch}
We give a quick sketch of the proof which results in the $0<\eta<2/7$
bound when we do not use the conductor lowering trick followed by the sketch to get
the improved range $0<\eta<2/5$.

The standard approximate functional equation argument in Section 1 reduces
our $L$-function to the analysis of the following sum
\[
  \sum_{n\sim \sqrt{P}M}\frac{\l_f(n)\chi(n)}{\sqrt{n}}.
\]
A second moment average over primitive characters leads us to needing to understand the sum
\[
    \sum_{n\sim\sqrt{P}M}\frac{\l_f(n)}{\sqrt{n}}
    \sum_{\substack{m\sim\sqrt{P}{M}\\m\equiv n\mod{M}}}
    \frac{\l_f(m)}{\sqrt{m}}.
\]
Using the delta symbol, we rewrite the above as
\[
  \sum_{0\neq|r|\leq\sqrt{P}}\sum_{n\sim\sqrt{P}M}
    \frac{\l_f(n)}{\sqrt{n}}\sum_{m\sim\sqrt{P}M}\frac{\l_f(m)}
    {\sqrt{m}}\d(m-n+rM,0),
\]
where the $r=0$ term can be bounded trivially. 
Applying the usual delta method and using Voronoi summation on both the $n$- and $m$-sums, 
one can obtain Kloosterman sums of the form
\begin{equation}
\label{oldmain}
  \sum_{0\neq|r|\leq\sqrt{P}}\frac{1}{Q}\sum_{q\leq Q}\frac{1}{q}
    \sum_{n\sim P}\frac{\l_f(n)}{\sqrt{n}}\sum_{m\sim P}\frac{\l_f(m)}{\sqrt{m}}
    S(rM,(n-m)\bar P;q).
\end{equation}
with  $Q=M^{1/2}P^{1/4}$ the square root of the size of the equation. 
Applying the Weil bound for each Kloosterman sum leads to the of second moment average being
bounded by
$\sqrt{P}\pair{\frac{P^7}{M^2}}^{1/8}$, 
such that $0<P<M^{2/7}$ is a range for subconvexity, as a bound of $\sqrt{P}$ would produce 
the convexity bound.

However, using the conductor lowering trick by instead
using $\d((n-m+rM)/P,0)$ with the condition $n-m+rM \equiv 0 \mod P$ followed by
Voronoi summation as in the previous case, 
one instead obtains the sum of Kloosterman sums
\begin{equation}
\label{newmain}
  \sum_{0\neq|r|\leq\sqrt{P}}\frac{1}{Q}\sum_{q\leq Q}\frac{1}{qP}
    \sum_{n\sim P}\frac{\l_f(n)}{\sqrt{n}}\sum_{m\sim P}\frac{\l_f(m)}{\sqrt{m}}
    S(rM,n-m;qP).
\end{equation}
The length of the $n$- and $m$-sums is the same as before, whereas we have additional saving 
$P$ in the denominator on the $q$-sum, which comes from detecting the congruence condition
$n-m+rM \equiv 0 \mod P$. The Weil bound for Kloosterman sums allows us to have a better bound by $P^{1/2}$. 
However, $Q$ has also changed and is
now $M^{1/2}/P^{1/4}$ as the equation in the delta symbol has changed and the division by
$Q$ -- which ends up a division by $Q^{1/2}$ because of the Weil bound -- means we only save
$P^{1/4}$. Dividing the previous bound by this gives $\sqrt{P}\pair{\frac{P^5}{M^2}}^{1/8}$.
The subconvexity range has been improved to $0<P<M^{2/5}$.

Let $\sigma$ be the element in $\Gamma_0(P)$. If $\textbf{a}$ is a cusp of $\Gamma_0(P)$, its stabilizer is defined by $\Gamma_{\textbf{a}}=\{ \sigma \in \Gamma_0(P) \;|\; \sigma \cdot \textbf{a}=\textbf{a}  \}$. In particular $\displaystyle \Gamma_{\infty}=\left\{ \pm \begin{pmatrix} 1 & b \\ & 1 \end{pmatrix}\; \middle| \;b \in \mathbb{Z}  \right\}$. Let $\textbf{a}$ and $\textbf{b}$ be two cusps of $\Gamma_0(P)$. Let $\sigma_{\textbf{a}} \in SL(2,\mathbb{R})$ be  a scaling matrix such that $\sigma_{\textbf{a}}\cdot\infty=\textbf{a}$ and $\sigma_{\textbf{a}}^{-1}\Gamma_{\textbf{a}}\sigma_{\textbf{a}}=\Gamma_{\infty}$. We define similarly for $\sigma_{\textbf{b}}$. The Kloosterman sum attached to two cusps $\textbf{a}$ and $\textbf{b}$ is defined in \cite{desiwa} as
\[
 S_{\textbf{a}\textbf{b}}(\alpha,\beta;c)=\sum_{\begin{pmatrix} a & * \\ c& d \end{pmatrix} \in \Gamma_{\infty} \backslash \sigma_{\textbf{a}}^{-1}\Gamma_0(P)\sigma_{\textbf{b}} \slash \Gamma_{\infty}} e \left( \frac{\alpha a+\beta d}{c} \right).
\]
Returning back to our analysis, we obtain the standard Kloosterman sum associated simply to the 
cusp at $\infty$ for the group $\Gamma_0(P)$ in \eqref{newmain}. 
Previously, $P$ was attached to $(n-m)$ and not the 
modulus $q$, which gave a Kloosterman sum associated with the cusps
at 0 and $\infty$ for $\Gamma_0(P)$ in \eqref{oldmain}.

Finally, we remark that we can improve the hybrid range further by using a Large Sieve 
inequality as in \cite{desiwa} to estimate the sums of Kloosterman sums instead of 
bounding them individually. Even using the Kuznetsov formula with a Large Sieve inequality 
would improve the subconvexity estimate if not the range of subconvexity. One could also 
introduce an amplification to improve the subconvexity bound. Since 
our purpose is simply to demonstrate the utility of the modified delta method in 
improving the range of subconvexity and simplifying the Kloosterman sum structure, we do not 
continue the argument in these ways.

\section
{Lemmas: summation formulas, the delta method}
Before beginning the proof of the theorem, we collect several lemmas
which will be used in the proof. The crux of the work is an application
of the delta method, which we state below.
The delta method was used in \cite{dfi}. This is a decomposition of the 
$\d$-symbol via a character sum. We use the version given by Heath-Brown 
\cite{heathbrown}.
\begin{lemma}[The delta method, \cite{heathbrown}]\label{dmethod}
For any $Q>1$ there exist $c_Q>0$ and a smooth function $g(x,y)$ defined on 
$(0,\infty) \times \mathbb{R}$, such that
\[
  \delta(n,0)=\frac{c_Q}{Q^2}\sum_{q=1}^{\infty}\underset{a\; \mathrm{mod} \;q}
  {{\sum}^{\star}}e\left(\frac{an}{q}\right)
  g\left(\frac{q}{Q},\frac{n}{Q^2}\right).
\]
The constant $c_Q$ satisfies $c_Q=1 + O_A(Q^{-A})$ for any $A > 0$. Moreover, 
$g(x,y) \ll x^{-1}$ for all $y$, and $g(x,y)$ is non-zero only for 
$x \leq \max\{1,2|y|\}$. If $x\leq 1$ and $2|y|\leq x$, then
\[
	x^i\pa{^i}{x^i}g(x,y) \ll x^{-1},\qquad \pa{}{y}g(x,y) = 0.
\]
If $2|y|>x$, then
\[
	x^iy^j\pa{^i}{x^j}\pa{^j}{y^j}g(x,y) \ll x^{-1}.
\]
\end{lemma}

Let $k$ be a positive integer. We recall elementary properties of 
$J$-Bessel functions as can be seen in \cite{whwa96}. 
The $J$-Bessel function is defined by
\[
  J_k(x)=e^{ix}W_k(x)+e^{-ix}\overline{W}_k(x)
\]
where
\[
  W_k(x)=\frac{e^{i(\frac{\pi}{2}k-\frac{\pi}{4})}}{\Gamma(k+\frac{1}{2})}
  \sqrt{\frac{2}{\pi x}}\int_{0}^{\infty}e^{-y}\left( y\left( 
  1+\frac{iy}{2x} \right) \right)^{k-\frac{1}{2}} dy.
\]
Applying
the asymptotic expansion for Whittaker functions $W$ for 
$x \gg 1$, we have
\[
  x^jW_k^{(j)}(x) \ll_{k,j} \frac{1}{(1+x)^{\frac{1}{2}}}
\]
with $j \geq 0$. Using the Taylor series expansion for $x \ll 1$,
we obtain the bound
\[
  x^jJ_k^{(j)}(x) \ll_{k,j} x^k 
\]
with $j \geq 0$. The $J$-Bessel function is used in the integral transform found in
Voronoi summation.
\begin{lemma}[Voronoi summation, \cite{kmvdk}]\label{voronoi}
Let $(a,q)=1$ and let $h$ be a smooth function compactly supported in
$(0,\infty)$. Let $f$ be a holomorphic newform of level $P$ and weight $k$. Set
$P_2:=P/(P,q)$. Then there exists a complex number $\eta$ of modulus 1 
(depending on $a,q,f$) and a newform $f^*$ of the same level $P$ and same 
weight $k$ such that
\[
	\sum_n\l_f(n)e\pair{n\frac{a}{q}}h(n) = \frac{2\pi\eta}{q\sqrt{P_2}}\sum_n
    \l_{f^*}(n)e\pair{-n\frac{\overline{aP_2}}{q}}\int_0^\infty h(y)J_{k-1}
    \pair{\frac{4\pi\sqrt{ny}}{q\sqrt{P_2}}}dy.
\]
\end{lemma}

In general, given $e(\bar a/q)$ the overline on the 
numerator $a$ indicates the multiplicative inverse modulo the denominator 
$q$. In order to truncate and bound the sums which result after Voronoi summation,
we use the following lemma.

\begin{lemma}\label{jbound}
	Let $k,M,P$ be positive integers with $k\geq2$ and let $r$ be a nonzero
    integer. Take $Q>1$ and $X,Y\geq1$. For any $a,b>0$, define
    \[
      J(a,b\;;c):= \int_0^\infty\int_0^\infty 
        \frac{1}{\sqrt{xy}}F\pair{\frac{x}{X},\frac{y}{Y}}
        g\pair{\frac{qc}{Q},\frac{x-y+rM}{PQ^2}}J_{k-1}(4\pi a\sqrt{x})
        J_{k-1}(4\pi b\sqrt{y})dxdy,
    \]
    where $g\pair{\frac{qc}{Q},\frac{x-y+rM}{PQ^2}}$ is the function in 
    lemma \ref{dmethod} and $F$ is a smooth function supported in
    $[1/2,5/2]\times[1/2,5/2]$ with partial derivatives satisfying
    \[
      x^iy^j\pa{^i}{x^i}\pa{^j}{y^j}F(x,y) \ll ZZ_x^iZ_y^j
    \]
    for some $Z>0$, $Z_x,Z_y\gg1$. We have
    \begin{align}
    	J(a,b\;;c) \ll Z\sqrt{XY}\frac{Q}{qc}
        \frac{1}{(1+a\sqrt{X})^{1/2}}
        \frac{1}{(1+b\sqrt{Y})^{1/2}} 
        \left[\frac{1}{a\sqrt{X}}\pair{{Z_x}+\frac{X}{qcQP}}\right]^i
        \left[\frac{1}{b\sqrt{Y}}\pair{{Z_y}
        +\frac{Y}{qcQP}}\right]^j.
        \label{jbound1}
    \end{align}
    Also,
    \begin{equation}
      J(a,b\;;c) \ll \frac{Z}{ab(1+a\sqrt{X})^{1/2}(1+b\sqrt{Y})^{1/2}} 
      \frac{Q}{qc}\min\left\{{Z_x b\sqrt{Y},{Z_y}a\sqrt{X}}\right\}
      (XY|r|MQqP)^{\e}.
      \label{jbound2}
    \end{equation}
\end{lemma}
\begin{proof}
Starting with the change of variables $x\mapsto xX$ and $y\mapsto yY$
and then integrating by parts and using the $J$-Bessel function bound
%\[
%  J_{k-1}(x) \ll \frac{x}{(1+x)^{3/2}}
%\]
gives \eqref{jbound1}. For example, integrating by parts once in the 
$x$ integral leads to
\[
  J(a,b\;;c) \ll Z\sqrt{XY}\frac{Q}{qc}\frac{1}{(1+a\sqrt{X})^{1/2}}
  \frac{1}{(1+b\sqrt{Y})^{1/2}}\left[\frac{1}{a\sqrt{X}}
  (Z_x+XI)\right],
\]
where
\[
  I := \underset{2|xX-yY+rM|>qQP}{\int_{1/2}^{5/2}\int_{1/2}^{5/2}} 
  \frac{1}{|xX-yY+rM|}\;dxdy. 
\]
The condition on the integral comes from where the $g$ function is nonzero.
However, we have $|xX-yY+rm|^{-1}\ll(qcQP)^{-1}$, which gives \eqref{jbound1}
for $i=1$ and $j=0$. Repeated integration by parts gives the same result for
higher values of $i$ and $j$.

For \eqref{jbound2}, we treat the integral $I$ differently. Let
$u=xX-yYrM$ to get
\[
  I \ll \frac{1}{X} \int^{5/2}_{1/2} \int^{(X+Y+|r|M)}_{qQP/2} 
  \frac{1}{u}dudy \ll \frac{(XY|r|MqQP)^{\e}}{X}.
\]
Doing the same thing with $i=0$ and $j=1$ and taking the minimum of the two bounds
finishes the proof.

\end{proof}

\section{Reduction of the second moment to shifted sums}
%The proof will follow from a second moment method 
%where we average over the primitive characters modulo $M$. 
Let $\e,\; 
\delta>0$ and choose any $X$ such that $\mathcal{Q}^{1/2-\delta}\leq X \leq 
\mathcal{Q}^{1/2+\e}$. Define
\[
  S_{f}(X):=\frac{1}{\varphi^\star(M)}\underset{\chi\mod{M}}{{\sum}^{\star}} 
    |L_{f \otimes \chi}(X)|^2,
\] 
where $L_{f \otimes \chi}(X)$ is given in \eqref{twisted-Lfunction}.
The proof of the reduction to shifted sums is similar to \cite{dfi}. We open 
the square and write the 
primitive characters in terms of Gauss sums to obtain
\[
  S_f(X) = \frac{1}{M\vp^\star(M)}\underset{\chi\mod{M}}{{\sum}^{\star}}
  \left| \sum_{b\mod M}\bar\chi(b)\sum_n\frac{\l_f(n)}{\sqrt{n}}
  e\pair{\frac{nb}{M}}
  \right|^2.
\]
Adding the nonprimitive characters to this sum produces the summation over all characters 
modulo $M$. By the orthogonality of multiplicative characters,
\[
  S_f(X) \leq \frac{\vp(M)}{M\vp^\star(M)}\underset{b\mod{M}}{{\sum}^{\star}}
  \left|
  \sum_n \frac{\l_f(n)}{\sqrt{n}} e\pair{\frac{nb}{M}}
  \right|^2.
\]
We extend the summation to all residue classes $M$
by adding the residues which are
not coprime with $M$, and open the square. Using the orthogonality of additive characters,
\[
  S_f(X) \ll M^{\e}
    \sum_n\frac{\l_f(n)}{\sqrt{n}}h\pair{\frac{n}{X}}
    \sum_{m\equiv n\mod{M}}\frac{\l_f(m)}{\sqrt{m}}h\pair{\frac{m}{X}}.
\]
Write $m=n+rM$ and note that $r\ll \frac{X}{M} \leq \frac{\mc Q^{\frac{1}{2}+\e}}{M}$. The diagonal term $m=n$ satisfies
\[
  \sum_{n}\frac{\lambda_{f}(n)^2}{n} h\left(\frac{n}{X}\right)^2 \ll 
  \mathcal{Q}^{\e}.
\]
We are left to consider the off-diagonal terms
\[
  R_f(X):= \sum_{0 \neq |r| \ll \frac{\mathcal{Q}^{{1}\slash{2}+\e}}{M}}
    \sum_{n} \frac{\lambda_{f}(n)}{\sqrt{n}} 
    h \left(\frac{n}{X}\right)\sum_{m=n+rM}
    \frac{\lambda_{f}(m)}{\sqrt{m}} h \left(\frac{m}{X}\right).
\]

\section{Treatment of shifted convolution sums}
Let $X,Y \geq 1$. Motivated by the reduction of the second moment problem to 
bounding the shifted convolution sums, we now take $f_1,f_2$ to newforms in $B^*_k(P)$ and consider 
\[
	S_{f_1,f_2}(X,Y) := \sum_{n}
    \sum_{m=n+rM}\frac{\l_{f_1}(n)\l_{f_2}(m)}
    {\sqrt{nm}}F\pair{\frac{n}{X},\frac{m}{Y}},
\]
where $r$ is a nonzero integer coprime to $P$ (valid in our specific 
application, where $|r|\ll P^{1/2+\e}$). In this section we establish 
theorem \ref{shifted-sums-thm}.

\subsection{Modified delta method}
We start with detecting the equation $m=n+rM$ in $S_{f_1,f_2}(X,Y)$. To do this,
we note that $n-m+rM=0$ is equivalent to $(n-m+rM)/P=0$
and $n-m+rM \equiv 0$ modulo $P$. Therefore,
\[
  S_{f_1,f_2}(X,Y)=\underset{n-m+rM \equiv 0 \mod {P}}{\sum \sum} 
  \frac{\l_{f_1}(n)\;\l_{f_2}(m)}{\sqrt{nm}} 
    F\pair{\frac{n}{X},\frac{m}
    {Y}} \delta\pair{\frac{n-m+rM}{P},0}.
\]
Using lemma \ref{dmethod} to detect $(n-m+rM)/P=0$ and a sum of additive
characters to detect the congruence gives
\[
\begin{split}
  S_{f_1,f_2}(X,Y)=& \sum_{n} \frac{\l_{f_1}(n)}
    {\sqrt{n}} \sum_{m} \frac{\l_{f_2}(m)}{\sqrt{m}}
    F\pair{\frac{n}{X},\frac{m}{Y}}\frac{1}{P}\sum_{b\mod{P}}
    e\left(\frac{b(n-m+rM)}{P}\right) \\
  & \times 
    \frac{c_Q}{Q^2}\sum_{q=1}^{\infty}\underset{a\mod{q}}
    {{\sum}^{\star}}e\left(\frac{a(n-m+rM)}{Pq}\right)g\left(\frac{q}{Q},
    \frac{n-m+rM}{PQ^2}\right) \\
  =& \frac{c_Q}{PQ^2} \sum_{q=1}^{\infty}\underset{a\mod{q}}{{\sum}^{\star}} 
    \sum_{b\mod{P}} e\pair{\frac{rM(a+bq)}{qP}}
    \sum_{n}\frac{\lambda_{f_1}(n)}{\sqrt{n}}e\pair{\frac{n(a+bq)}{qP}} \\
  & \times  \sum_{m} \frac{\lambda_{f_2}(m)}{\sqrt{m}}
    e\pair{-\frac{m(a+bq)}{qP}} F\pair{\frac{n}{X},\frac{m}{Y}} g\left(\frac{q}
    {Q},\frac{n-m+rM}{PQ^2}\right).
\end{split}
\]
From lemma \ref{dmethod}, $g\pair{\frac{q}{Q},\frac{n-m+rM}{PQ^2}} \neq 0$ for 
$\frac{q}{Q} \leq \max \left\{1,\frac{2|n-m+rM|}{PQ^2}\right\}$. We want to
choose $Q$ such that the max is always 1. To do this, note that
\[
  2\frac{|n-m+rM|}{PQ^2} \leq 2\frac{X+Y+|r|M}{PQ^2}.
\]
For there to be any solutions to $m-n+rM=0$, we need 
$|r|\leq (X+Y)/M$. Therefore it is sufficient to choose $Q$
such that
\[
  8\frac{\max\set{X,Y}}{PQ^2} \leq 1.
\]
Set $Q^2:=8\frac{\max\set{X,Y}}{P}$ so that the outer $q$-sum only extends to $Q$. We 
write $\gamma=a+bq$. Since the $a$- and $b$-sums yield a complete set of residue 
$\gamma$ modulo $qP$ with $(\gamma,q)=1$, $S_{f_1,f_2}(X,Y)$ reduces to
\[
\begin{split}
  \frac{c_Q}{PQ^2} \sum_{q=1}^Q\;
  	\sum_{\substack{\gamma \mod{qP}\\(\gamma,q)=1}}
    e\left(\frac{rM\gamma}{qP}\right)
    \sum_{n}
      \frac{\l_{f_1}(n)}{\sqrt{n}}e\left(\frac{n\gamma}{qP}\right) 
      \sum_{m} \frac{\l_{f_2}(m)}{\sqrt{m}}
    e\left(-\frac{m\gamma}{qP}\right)
    F\left(\frac{n}{X},\frac{m}{Y}\right)
    g\left(\frac{q}{Q},\frac{n-m+rM}{PQ^2}\right).
\end{split}    
\]

\subsection{Voronoi summation}

We apply Voronoi summation to $S_{f_1,f_2}(X,Y)$ in the $m$-sum, then in the 
$n$-sum. Since we are dealing with forms of prime level $P$, we break the 
$q$-sum above as $S_{f_1,f_2}(X,Y)=\mathcal{S}+\mathcal{T}$, 
where $\mathcal{S}$ is the sum over $(q,P)=1$ and $\mathcal{T}$ is the sum over 
$P\; |\; q$.
%noting that these cases are exhaustive since $P$ is prime.

\subsubsection{$(q,P)=1$ case}
Assume that $(q,P)=1$. We split the $\gamma$-sum in $\mathcal{S}$
as $\mathcal{S}=\mathcal{S}_1+\mathcal{S}_2$, where $\mathcal{S}_1$
is the sum over $\gamma$ for which $(P,\gamma)=1$ and $\mathcal{S}_2$ is the sum over $\gamma$ for which 
$P\;|\;\gamma$. When $(P,\gamma)=1$, we get the sum
\[
\begin{split}
  \mathcal{S}_1:=&\frac{c_Q}{PQ^2} \underset{(q,P)=1}{\sum_{q=1}^Q}
    \underset{\gamma\mod{qP}}{{\sum}^{\star}}
    e\left(\frac{rM\gamma}{qP}\right)
    \sum_{n}
      \frac{\l_{f_1}(n)}{\sqrt{n}}e\left(\frac{n\gamma}{qP}\right) \\
  &\qquad\times \sum_{m} \frac{\l_{f_2}(m)}{\sqrt{m}}
    e\left(-\frac{m\gamma}{qP}\right)
    F\left(\frac{n}{X},\frac{m}{Y}\right)
    g\left(\frac{q}{Q},\frac{n-m+rM}{PQ^2}\right).
\end{split}
\]
For each $q$, applying lemma \ref{voronoi} first in the $n$-sum 
and then in the $m$-sum results in
\[
    \frac{1}{P^2q^2}
    \sum_{n}\l_{f_1^*}(n)e\pair{\frac{-n\o{\gamma}}{
    Pq}}
    \sum_{m}\l_{f_2^*}(m)e\pair{\frac{m\o{\gamma}}{Pq}}
    J \pair{\frac{\sqrt{n}}{Pq},\frac{\sqrt{m}}{
    Pq}\;;1},
\]
where $J$ is again the function given in lemma \ref{jbound}. 
This produces a Kloosterman sum modulo $Pq$:
\[
   \frac{1}{Q^2P^{3}}\underset{(q,P)=1}
  {\sum_{q=1}^Q} \frac{1}{q^2} \sum_{n} \lambda_{f_1^*}(n)
  \sum_{m}
  \lambda_{f_2^*}(m) \;J \pair{\frac{\sqrt{n}}{Pq},\frac{\sqrt{m}}{Pq}\;;1}
  S(rM,m-n;Pq).
\]
Using the first bound of the $J$-function in lemma \ref{jbound} allows 
us to truncate the $n$- and $m$-sums to the ranges
\[
   n \leq T_1:=\frac{P^2q^2}{X}\left(Z_x+\frac{X}{qQP}\right)^2,
   \quad m \leq T_2:=\frac{P^2q^2}{Y}\left(Z_y+\frac{Y}{qQP} 
   \right)^2.
\]
Applying the Weil bound and the second bound in lemma 
\ref{jbound}, we see that
\[
\begin{split}
\mc S_1 \ll \frac{{\mathcal{Q}}^{\e}}{Q^2P^{3}}&\underset{(q,P)=1}
     {\sum_{q=1}^Q} \frac{1}{q^2}\sum_{n \leq T_1}\sum_{m \leq T_2}
     \frac{Z}{\frac{\sqrt{n}}{Pq}\frac{\sqrt{m}}{Pq}\left(1+\frac{\sqrt{nX}}{Pq}
     \right)^{\frac{1}{2}}
     \left(1+\frac{\sqrt{mY}}{Pq}\right)^{\frac{1}{2}}}\\
  & \times \frac{Q}{q}\min\left\{Z_x\frac{\sqrt{mY}}{Pq},Z_y
     \frac{\sqrt{nX}}{Pq}\right\} (rM,m-n,Pq)^{\frac{1}{2}} 
     (Pq)^{\frac{1}{2}}. 
\end{split}
\]
We bound the minimum by taking the geometric mean and simplify to
\[
\begin{split}
  \mc S_1 &\ll \mathcal{Q}^{\e}Z(Z_xZ_y)^{\frac{1}{2}} 
       \frac{1}{QP^{\frac{1}{2}}} \underset{(q,P)=1}{\sum_{q=1}^Q}  
        \frac{1}{q^{\frac{1}{2}}} \sum_{n \leq T_1} \sum_{m \leq T_2} 
        \frac{1}{(nm)^{\frac{1}{2}}}(rM,m-n,Pq)^{\frac{1}{2}} 
       \\
        & \ll\mathcal{Q}^{\e}Z(Z_xZ_y)^{\frac{1}{2}} \frac{1}
        {QP^{\frac{1}{2}}} \underset{(q,P)=1}{\sum_{q=1}^Q}  
        \frac{1}{q^{\frac{1}{2}}}(T_1T_2)^{\frac{1}{2}} 
        (rM,m-n,Pq)^{\frac{1}{2}}  
        \\
        &\ll \mathcal{Q}^{\e}Z(Z_xZ_y)^{\frac{1}{2}}
        \frac{1}{(XY)^{\frac{1}{2}}}\frac{P^{\frac{3}{2}}}{Q} 
        \underset{(q,P)=1}{\sum_{q=1}^Q} q^{\frac{3}{2}} 
        \left(Z_x+\frac{X}{qQP}\right)\left(Z_y+\frac{Y}{qQP}\right)
        (rM,m-n,Pq)^{\frac{1}{2}}.
\end{split}
\]
Notice that $(rM,m-n,Pq) \leq (rM,Pq) \leq (rM,q)$, since $r$ and $P$ are coprime, 
and $P$ does not divide $M$. Rewriting $q$ as $qd$ with $d=(rM,q)$ gives
\[
   \begin{split}
   \mc S_1 & \ll \mathcal{Q}^{\e}Z(Z_xZ_y)^{\frac{1}{2}}
        \frac{1}{(XY)^{\frac{1}{2}}}\frac{P^{\frac{3}{2}}}{Q} 
        \sum_{d | rM} d^{\frac{1}{2}}
        \underset{q \leq Q/d}{\sum} 
        (qd)^{\frac{3}{2}} \left(Z_x+\frac{X}{qdQP}\right)
        \left(Z_y+\frac{Y}{qdQP}\right)  \\
        & \ll \mathcal{Q}^{\e}Z(Z_xZ_y)^{\frac{1}{2}}
        \frac{1}{(XY)^{\frac{1}{2}}} P^{\frac{3}{2}}Q^{\frac{3}{2}}
        \sum_{d | rM}\frac{1}{d^{\frac{1}{2}}}
        \left(Z_x+\frac{X}{Q^2P} \right) \left(Z_y+\frac{Y}{Q^2P} \right)\\
   & \ll \mathcal{Q}^{\e}r^{\e}Z(Z_xZ_y)^{\frac{1}{2}}
        \frac{1}{(XY)^{\frac{1}{2}}} P^{\frac{3}{2}}Q^{\frac{3}{2}}
        \left(Z_x+\frac{X}{Q^2P} \right) \left(Z_y+\frac{Y}{Q^2P} 
        \right)\\
       & \ll  \mathcal{Q}^{\e}r^{\e}Z(Z_xZ_y)^{\frac{1}{2}}
        \frac{1}{(XY)^{\frac{1}{2}}} P^{\frac{3}{2}}Q^{\frac{3}{2}}
        \left(\max \left\{Z_x,Z_y \right\}+ \frac{\max 
        \left\{X,Y\right\}}{Q^2P}\right)^{2}.
   \end{split}
\]
With $Q^2=8\frac{\max \{X, Y\}}{P}$, we obtain
\[
  \mc S_1 \ll \mc{Q}^\e r^\e Z(Z_xZ_y)^{\frac{1}{2}}\max\set{Z_x,Z_y}^2
    P^{\frac{3}{4}}\frac{\max\set{X,Y}^{\frac{3}{4}}}{
    (XY)^{\frac{1}{2}}}.
\]
This matches the bound in the statement of theorem \ref{shifted-sums-thm}.

Next we deal with the sum where $P|\gamma$.
\[
\begin{split}
  \mathcal{S}_2:=&\frac{c_Q}{PQ^2} \underset{(q,P)=1}{\sum_{q=1}^Q}\;
  	\sum_{\substack{\gamma \mod{qP}\\(\gamma,q)=1\\ P|\gamma}}
    e\left(\frac{rM\gamma}{qP}\right)
    \sum_{n}
      \frac{\l_{f_1}(n)}{\sqrt{n}}e\left(\frac{n\gamma}{qP}\right) \\
  &\qquad\times \sum_{m} \frac{\l_{f_2}(m)}{\sqrt{m}}
    e\left(-\frac{m\gamma}{qP}\right)
    F\left(\frac{n}{X},\frac{m}{Y}\right)
    g\left(\frac{q}{Q},\frac{n-m+rM}{PQ^2}\right)\\
    =&\frac{c_Q}{PQ^2} \underset{(q,P)=1}{\sum_{q=1}^Q}
  	\underset{\gamma\mod{q}}{{\sum}^{\star}}
    e\left(\frac{rM\gamma}{q}\right)
    \sum_{n}
      \frac{\l_{f_1}(n)}{\sqrt{n}}e\left(\frac{n\gamma}{q}\right) \\
  &\qquad\times \sum_{m} \frac{\l_{f_2}(m)}{\sqrt{m}}
    e\left(-\frac{m\gamma}{q}\right)
    F\left(\frac{n}{X},\frac{m}{Y}\right)
    g\left(\frac{q}{Q},\frac{n-m+rM}{PQ^2}\right).
\end{split}
\]
Applying Voronoi summation in the $m$- and $n$-sums gives
\[
  \frac{1}{P^2Q^2}\sum_{q=1}^Q\frac{1}{q^2}
  \sum_{n}\l_{f_1^*}(n)\sum_m\l_{f_2^*}(m)J\pair{\frac{\sqrt{n}}{q\sqrt{P}},
  \frac{\sqrt{m}}{q\sqrt{P}};1}
  S(rM,(m-n)\bar{P};q).
\]
The $n$- and 
$m$-sums can be truncated to
\[
   n \leq T_1:=\frac{Pq^2}{X}\left(Z_x+\frac{X}{qQP}\right)^2,
   \quad m \leq T_2:=\frac{Pq^2}{Y}\left(Z_y+\frac{Y}{qQP}\right)^2.
\]
Repeating the estimation process used for $\mc S_1$ with these parameters, 
one is led to
\[
  \mc S_2 \ll \mc{Q}^\e r^\e Z(Z_xZ_y)^{\frac{1}{2}}\max\set{Z_x,Z_y}^2
    \frac{1}{P^{\frac{3}{4}}}\frac{\max\set{X,Y}^{\frac{3}{4}}}{(XY)^{\frac{1}{2}}}.
\]
This is better than the bound of theorem \ref{shifted-sums-thm}.

\subsubsection{$P\; |\; q$ case}
Finally, observing that $\gamma$ is coprime with $Pq$, 
the remaining sum over $q$, when $P\;|\;q$ is
%We rewrite $q$ as $qP$. 
%Considering the support of $g$ and given our choice 
%of $Q$, we reduce the sum $S_{f_1,f_2}(X,Y)$ to the analysis of
\begin{align*}
	\mathcal{T} &:= \frac{c_Q}{PQ^2}\underset{P|q}{\sum_{q=1}^Q}
    \underset{\gamma\mod{Pq}}{{\sum}^\star}
    e\pair{\frac{rM\gamma}{Pq}}\sum_{n}\frac{\l_{f_1}(n)}{\sqrt{n}}
    e\pair{\frac{n\gamma}{Pq}} \\
    &\qquad \times \sum_{m}
    \frac{\l_{f_2}(m)}{\sqrt{m}}e\pair{-\frac{m\gamma}{
    Pq}}F\pair{\frac{n}{X},\frac{m}{Y}}g\pair{\frac{q}{Q},
    \frac{n-m+rM}{PQ^2}}.
\end{align*}
%Noting that $a+bqP$ is coprime with $P^2q$ so that lemma 
%\ref{voronoi} is applicable, we use Voronoi summation in the $m$-sum then the $n$-sum. Since the 
%$a$- and $b$-sums cover a complete set of residues modulo $P^2q$, we obtain
After we rewrite $q$ as $qP$, Voronoi summation in the $m$-sum then the $n$-sum gives
\[
  \frac{1}{Q^2P^5}
  \sum_{q\leq Q/P}\frac{1}{q^2} \sum_{n}
  \sum_{m}\l_{f_1^*}(n)\l_{f_2^*}(m) 
  J\pair{\frac{\sqrt{n}}{P^2q},\frac{\sqrt{m}}{P^2q}\;;P} S(rM,m-n;P^2q).
\]
Lemma \ref{jbound} implies that we can truncate the $n$- and $m$-sums to the
ranges
\[
  n \leq T_1:=\frac{P^{4}q^2}{X}\left(Z_x+\frac{X}{qQP^2}\right)^2,\quad 
  m \leq T_2:=\frac{P^{4}q^2}{Y}\left(Z_y+\frac{Y}{qQP^2}\right)^2.
\]
Using the Weil bound 
for Kloosterman sums as in the previous case, we see that
\[
  \mc T \ll \mc Q^\e r^\e Z(Z_xZ_y)^{\frac{1}{2}}\max\set{Z_x,Z_y}^2
    \frac{1}{P^{\frac{1}{4}}}\frac{\max\set{X,Y}^{\frac{3}{4}}}{(XY)^{\frac{1}{2}}}.
\]
This bound on $\mc T$ is better than the bound on $\mc S_1$.
This completes the proof of theorem \ref{shifted-sums-thm}.

\vskip.1in
\noindent
\textbf{Acknowledgments.} This work was completed as an extension of a topics course 
given by Roman Holowinsky at The Ohio State University in Spring 2015. We are very grateful to Roman Holowinsky for a careful reading of the draft. The authors  
would also like to thank Fei Hou for identifying an error in an earlier version of 
this manuscript. 
%We would like
%to thank Fei Hou for making suggestions leading to some corrections. 
%We would like to 
%thank Roman Holowinsky for his encouragement, 
%invaluable comments, and many careful reading of this paper. It is a 
%pleasure to thank Ritabrata Munshi for suggesting this idea to us.

\nocite{} 
\bibliographystyle{amsplain}
\bibliography{hybrid-subconvexity-gl2-gl1}

\end{document}